\documentclass[11pt]{amsart}

\usepackage[T1]{fontenc}
\usepackage[utf8]{inputenc}
\usepackage{lmodern}
\usepackage{amsmath,amssymb,amsthm,mathtools}
\usepackage{enumitem}
\usepackage{microtype}
\usepackage[margin=1in]{geometry}
\usepackage[colorlinks=true,linkcolor=blue,citecolor=blue,urlcolor=blue]{hyperref}

\newtheorem{theorem}{Theorem}[section]
\newtheorem{conjecture}[theorem]{Conjecture}
\newtheorem{proposition}[theorem]{Proposition}
\newtheorem{lemma}[theorem]{Lemma}

\theoremstyle{remark}
\newtheorem{remark}[theorem]{Remark}

\newcommand{\F}{\mathbb F}
\newcommand{\Gm}{\mathbb G_{\mathrm m}}

\newcommand{\Pj}{\mathbb P}
\newcommand{\Ql}{\overline{\mathbb Q}_{\ell}}
\newcommand{\Tr}{\operatorname{Tr}}
\newcommand{\Nr}{\operatorname{Nr}}
\newcommand{\Frob}{\operatorname{Frob}}
 
\newcommand{\Swan}{\operatorname{Swan}}
\newcommand{\Kl}{\operatorname{Kl}}

\newcommand{\TInd}{\operatorname{TInd}}

\newcommand{\Gal}{\operatorname{Gal}}
\newcommand{\Sym}{\operatorname{Sym}}

\newcommand{\cF}{\mathcal F}
\newcommand{\cK}{\mathcal K}
\newcommand{\cL}{\mathcal L}
\newcommand{\cG}{\mathcal G}
\newcommand{\cH}{\mathcal H}

\newcommand{\geom}{\mathrm{geom}}

\def\GL{\operatorname{GL}}
\def\der{\operatorname{der}}
\def\Lie{\operatorname{Lie}}

\newcommand{\Nm}{\operatorname{N}}

\title{Matrix Kloosterman sums and product-trace estimates for semisimple algebras
}
\author{Xuejun Guo}
\thanks{Xuejun Guo, School of Mathematics, Nanjing University,
Nanjing 210093, China.
\href{mailto:guoxj@nju.edu.cn}{guoxj@nju.edu.cn}.}

\author{Chen Lin}
\thanks{Chen Lin, School of Mathematics, Nanjing University,
Nanjing 210093, China.
\href{mailto:chen.lin@smail.nju.edu.cn}
{chen.lin@smail.nju.edu.cn}.}

\author{Chenhao Tang}
\thanks{Chenhao Tang (Corresponding Author), Morningside Center of Mathematics,
Academy of Mathematics and Systems Science,
Chinese Academy of Sciences; University of Chinese Academy of Sciences,
Beijing 100190, China.
\href{mailto:tangchenhao25@mails.ucas.ac.cn}
{tangchenhao25@mails.ucas.ac.cn}.}
\date{}



\begin{document}

\begin{abstract}
Let $k=\F_q$, $E=\F_{q^n}$ and $\Tr=\Tr_{E/k}$. For $r\ge 2$, $a\in k^{\times}$ and $x\in E^{\times}$, let $\Nm(E,r,x,a)$ be the number of $r$-tuples $(x_1,\cdots,x_r)$ in $(E^{\times})^r$ satisfying
$x_1\cdots x_r=x$ and $\Tr(x_1+\cdots+x_r)=a$. We prove 
$\left|\Nm(E,r,x,a)-\left((q^n-1)^{r-1}+(-1)^r\right)/q\right|\le (r^n-1) q^{\frac{(r-1)n-1}{2}}$. This proves the square-root estimate predicted in Wan's conjecture and generalizes a previous result of Moisio and Wan. For a finite semisimple algebra
$B=\prod\limits_{i=1}^s M_{d_i}(\F_{q^{n_i}})$ over $k$ and a regular element $x\in B^{\times}$, the same method combined with Zelingher's formula leads to analogous square-root estimates.
\end{abstract}

\maketitle

\section{Introduction and statement of results}
The problem of prescribing the trace and the norm simultaneously is a classical counting problem over finite fields. For a field extension $\F_{q^n}/\F_q$, Katz obtained a square-root estimate using exotic Kloosterman sheaves \cite{KatzSoto}; related estimates were obtained by Moisio and by Moisio-Wan using hyper-Kloosterman sums and the
Hasse-Davenport relation \cite{MoisioKloosterman,MoisioWan}. In the geometric case $B=\F_{q}^n$, the same problem is the point-counting problem for
a family of toric Calabi-Yau hypersurfaces; see
\cite{RojasLeonWanMoments,WanLectures}. The field and geometric cases were subsequently unified and extended to arbitrary finite \'etale algebras by
Lin and Wan \cite{LinWanEtale}. The trace-norm $L$-functions were also studied by Rojas-Le\'on
\cite{RojasLeonTraceNorm}.

These commutative results form the background for Wan's recent extension to finite semisimple algebras \cite{WanNormTrace}. The key point of that work is
a reduction from a general semisimple algebra to a split finite \'etale algebra. The reduction combines the Hasse-Davenport relation with Eichler's formula for Gauss sums over general linear groups
\cite{EichlerGauss}; related proofs and formulations of the matrix Gauss sum identity can be found in \cite{LamprechtGauss,KimGauss,LiHuGauss}. It yields
a square-root estimate for the number of elements with prescribed reduced trace and reduced norm, as well as a corresponding estimate for Kloosterman
sums over finite semisimple algebras.

The product-trace problem considered here was proposed in the final section of \cite{WanNormTrace}. Its exponential-sum counterpart is the classical
hyper-Kloosterman sum when $B=\F_q$, whereas for a matrix algebra it is the matrix Kloosterman sum studied by Zelingher \cite{ZelingherMatrix}. For a
regular matrix, Zelingher's formula expresses this sum in terms of Frobenius traces on symmetric powers of the Kloosterman sheaf. Thus, Katz's theory of
Kloosterman sheaf \cite{KatzGKM}, together with Zelingher's formula, is
directly suited to the product-trace question. The purpose of the present paper is to prove the expected square-root estimate first for field extensions and then for arbitrary finite semisimple algebras.

We now formulate the product-trace problem studied in this paper. 

Let $p$ be a prime and $\F_q$ the finite field of $q$ elements with characteristic $p$. For any positive integers $m$ and $d$, let $\F_{q^m}$ denote the extension of $\F_q$ with degree $m$ and $M_{d}(\F_{q^m})$ denote the $d\times d$ matrix algebra over $\F_{q^m}$. By the Artin-Wedderburn theorem, a  finite semisimple algebra over $\F_q$ has the form
\begin{equation}\label{B}
B = M_{d_1}(\F_{q^{n_1}}) \times \cdots \times M_{d_s}(\F_{q^{n_s}})
\end{equation}
with dimension $N=n_1 d_{1}^2+\cdots + n_s d_{s}^2$. 
Denote the reduced trace and the reduced norm of $B$ by $\Tr_B$ and $\Nr_B$ respectively. For any given $a\in \F_q$ and $b\in \F_q^{\times}$, let
\[
\Nm_B(a, b)= \# \{ x \in B^{\times} \mid \Tr_B(x)=a,~\Nr_B(x)=b\}.
\]
Wan proved the following theorem in \cite{WanNormTrace} using the Hasse-Davenport relation of Kloosterman sums. 

\begin{theorem}[Wan, Theorem 1.1 of \cite{WanNormTrace} ]\label{wanthm}
Let $B$ be a finite semisimple algebra over $\F_q$ as given in \eqref{B}. For $a \in \F_q$ and $b\in\F_q^{\times}$, we have
\[
\left|\Nm_B (a, b) - \left( \frac{|B^{\times}|}{q(q-1)} + \frac{(-1)^{\sum_{i=1}^s d_i}q^{\sum_{i=1}^s n_i \binom{d_i}{2}}}{q} \right) \right| \le\left(\sum_{i=1}^s d_i n_i -1\right)q^{\frac{N-2}{2}}.
\]       
\end{theorem}

For $x \in B^{\times}$, $a \in \F_q$ and $r\ge 2$, let 
\begin{equation}\label{nbrxa}
\Nm(B,r,x,a) = \#\{ (g_1,\cdots, g_r) \in (B^{\times})^r \mid g_1\cdots g_r = x,~\Tr_B(g_1+\cdots +g_r) =a\}.
\end{equation}
Wan proposed the following conjecture in \cite{WanNormTrace}. 

\begin{conjecture}[Wan, Conjecture 4.3 of \cite{WanNormTrace} ]\label{conj2} Let $B$ be a semisimple algebra of dimension $N$ over $\F_q$ as in \eqref{B}. For $a\in \F_q^{\times}$,
regular $x \in B^{\times}$ and integer $r\ge 2$, we have
\[
\left|\Nm(B,r,x,a) - \frac{|B^{\times}|^{r-1}}{q}\right| \le r^{\sum_{i=1}^s d_i n_i} q^{\frac{(r-1)N-1}{2}}.
\]
\end{conjecture}

Note that $n_1, \cdots, n_s$ are missing in the original form of Conjecture 4.3 of \cite{WanNormTrace}. Wan told the first author of this paper that this is a typo. We therefore use the corrected form above. We will give a counterexample in Section~\ref{counterexample} to show that this modification is necessary.

We first consider this conjecture when $B$ is a field. 
Let $k=\F_q,~E=\F_{q^n}$
and $r\ge 2,~x\in E^{\times},~a\in k^{\times}$. Put
\begin{equation}\label{eq:mxa-def}
        m_x(a)=\#\left\{(z_0,\cdots,z_{n-1})\in(\bar{k}^{\times})^n\mid
        z_{j}^r=x^{q^j},~r(z_0+\cdots+z_{n-1})=a\right\}.
\end{equation}
In particular, if $p=\operatorname{char}k$
divides $r$, then
\begin{equation}\label{eq:m-zero-p-divides-r}
        m_x(a)=0.
\end{equation}

\begin{theorem}\label{main thm}
For every prime power $q$, every $n\ge 1$, every $r\ge 2$, every
$a\in k^{\times}$, and every $x\in E^\times$, one has
\begin{equation}\label{eq:main}
\left|\Nm(E,r,x,a)-\frac{(q^n-1)^{r-1}}{q}-\frac{(-1)^r}{q}\right|\le(r^n-m_x(a)-1)q^{\frac{(r-1)n-1}{2}}.
\end{equation}
In particular
\begin{equation}\label{eq:main-basic}
\left|\Nm(E,r,x,a)-\frac{(q^n-1)^{r-1}}{q}-\frac{(-1)^r}{q}\right|\le (r^n-1) q^{\frac{(r-1)n-1}{2}}.
\end{equation}
\end{theorem}

\begin{remark}
The estimate \eqref{eq:main-basic} is a generalization of the main result in \cite{MoisioWan}. In fact, using the Davenport-Hasse relation, the authors reduced Theorem 1.2 of \cite{MoisioWan} to Theorem 2.2 of \cite{MoisioWan}. The notation $N(u)$ of \cite{MoisioWan} is exactly $\Nm(k,n+1,u,1)$ in our convention, and the estimate in Theorem 2.2 of \cite{MoisioWan} is exactly \eqref{eq:main-basic}.
\end{remark}

Sections~\ref{sec:fourier}-\ref{sec:cohomological-estimates} are devoted to proving Theorem~\ref{main thm}, and the proof will be finished in Proposition~\ref{prop:S-swan-estimate}. The proof starts with a Fourier reduction to a one-variable sum over $k^{\times}$ with each term being a twisted hyper-Kloosterman sum, which can be written as the $\Frob_E$-trace on Katz's Kloosterman sheaf. Through tensor induction we construct a new sheaf and descend the $\Frob_E$-trace to a $\Frob_k$-trace. This enables us to use the trace formula. The key step is to show the vanishing of $H_c^2$. The main ingredients are Katz's global monodromy theorem and the Goursat-Kolchin-Ribet criterion.

We generalize Theorem~\ref{main thm} to the semisimple case in Section~\ref{sec:ss-regular-sheaf}
and Section~\ref{sec:ss-cohomology-proof}. The notations used below will be explained at the beginning of Section~\ref{sec:ss-regular-sheaf}. Let $B$ be as in \eqref{B}, $r\ge2,~a\in k^{\times}$. We denote $M=\sum_{i=1}^s n_i d_i$. For a regular element
$x=(x_i)_i\in B^{\times}$, assume that the rational canonical form (see Proposition~\ref{prop:ss-primary-decomposition}) of each $x_i$ is
\[
x_i\sim\operatorname{diag}\left(J_{b_{i1}}(y_{i1}),\cdots,
 J_{b_{i t_i}}(y_{i t_i})\right),
\]
where $y_{i j}$ is regular elliptic of degree $a_{i j}$ over
$\F_{q^{n_i}}$, and put $f_{i j}=n_i a_{i j}$. We can also define a number $m_{B,x}(a)$ analogously to $ m_x(a)$, see the beginning of Section~\ref{sec:ss-cohomology-proof}.

\begin{theorem}\label{thm:ss-main}
Let $B$ be as in \eqref{B}, let $r\ge 2$, let
$a\in k^{\times}$, and let $x\in B^{\times}$ be regular. Then
\begin{equation}\label{eq:ss-main}
\left|\Nm(B,r,x,a)-\frac{|B^{\times}|^{r-1}}{q}
 +\frac{\varepsilon_B q^{A_B}}{q}\right|\le
\left(R_{B,x}-m_{B,x}(a)-1\right)q^{\frac{(r-1)N-1}{2}},
\end{equation}
where
\[
R_{B,x}=\prod_{1\le i\le s,1\le j\le t_i}\binom{b_{i j}+r-1}{b_{i j}}^{f_{i j}},\quad A_B=\frac{(r-1)(N-M)}{2},\quad \varepsilon_B=(-1)^{(r-1)\sum_{i=1}^s d_i}.
\]
\end{theorem}

\begin{remark}
(a)  When $B=\F_{q^n}$, estimate \eqref{eq:ss-main} specializes exactly to \eqref{eq:main}. See Remark~\ref{rem:ss-field-specialization}. Thus, Theorem~\ref{thm:ss-main} recovers Theorem~\ref{main thm}.\\
(b) One can check $R_{B,x}\le r^M =r^{\sum_{i=1}^s n_i d_i}$. In particular, Conjecture~\ref{conj2} can be deduced from Theorem~\ref{thm:ss-main}.
\end{remark}

The proof of Theorem~\ref{thm:ss-main} will be finished in Proposition~\ref{prop:ss-trace-estimate}. Compared with Theorem~\ref{main thm}, the new ingredient of the proof is Zelingher's formula for
regular matrix Kloosterman sums, which expresses each simple-factor sum in
terms of Frobenius traces on symmetric powers of the Kloosterman sheaf.

Throughout the paper, $\ell$ is a prime different
from $p$, and a fixed embedding $\Ql\hookrightarrow\mathbb{C}$ is used to take
complex absolute values of $\ell$-adic trace sums.  We use geometric Frobenius.

\section{Fourier reduction}\label{sec:fourier}

Let $k=\F_q,~E=\F_{q^n}$. Fix a non-trivial additive character
\[
\psi:k\longrightarrow \Ql^{\times}
\]
and put $\Psi=\psi\circ\Tr_{E/k}$. For $y\in E^{\times}$, define the $r$-variable hyper-Kloosterman sum
\begin{equation}\label{eq:Kl-def}
    \Kl_{r,E,\psi}(y)=\sum_{u_1,\cdots,u_{r-1}\in E^{\times}}\Psi\left(u_1+\cdots+u_{r-1}+\frac{y}{u_1\cdots u_{r-1}}\right).
\end{equation}

\begin{lemma}\label{lem:fourier}
For $a\in k^{\times}$ and $x\in E^{\times}$, one has
\begin{equation}\label{eq:fourier-N}
    \Nm(E,r,x,a)=\frac{(q^n-1)^{r-1}}{q}+\frac{1}{q} S(E,r,x,a),
\end{equation}
where
\begin{equation}\label{eq:S-def}
    S(E,r,x,a)=\sum_{t\in k^{\times}}\psi(-at) \Kl_{r,E,\psi}(t^r x).
\end{equation}
\end{lemma}

\begin{proof}
For $z\in E$, orthogonality relation of additive characters gives
\[
    1_{\{z\in E \,\mid\, \Tr_{E/k}(z)=a\}}
    =\frac{1}{q}\sum_{t\in k}\psi\left(t\left(\Tr_{E/k}(z)-a\right)\right)=\frac{1}{q}\sum_{t\in k}\psi(-at)\Psi(tz).
\]
Apply this identity to
\[
z=x_1+\cdots+x_{r-1}+\frac{x}{x_1\cdots x_{r-1}}
\]
and sum over $(x_1,x_2,\cdots,x_{r-1})\in(E^{\times})^{r-1}$.
The term $t=0$ contributes to $(q^n-1)^{r-1}/q$. For $t\neq 0$, the change of
variables $u_i=tx_i$ is a bijection of $(E^{\times})^{r-1}$ and gives
\[
    t\left(x_1+\cdots+x_{r-1}+\frac{x}{x_1\cdots x_{r-1}}\right)=u_1+\cdots+u_{r-1}+\frac{t^r x}{u_1\cdots u_{r-1}}.
\]
This gives \eqref{eq:fourier-N}.
\end{proof}

It remains to estimate $S(E,r,x,a)$.

\section{Tensor induction of Kloosterman sheaves}\label{sec:tensor-induction}

The purpose of this section is to construct a lisse $\ell$-adic sheaf on ${\Gm}_k$ whose local trace gives $S(E,r,x,a)$.

Let $\cK=\cK_{r,\psi}$ be Katz's rank $r$ Kloosterman sheaf; see \cite[Theorem~4.1.1]{KatzGKM}. It is a lisse $\ell$-adic sheaf on ${\Gm}_k$, pure of weight $r-1$, tame at $0$ and totally wild at $\infty$ with all upper breaks equal to $\frac{1}{r}$, satisfying that for any finite extension $L/k$ and any $y\in L^{\times}$, one has
\begin{equation}\label{eq:Kl-sheaf-trace}
    \operatorname{tr}(\Frob_{L,y}\mid \cK_{\bar y})=(-1)^{r-1}\Kl_{r,L,\psi}(y),
\end{equation}
where $\Frob_{L,y}$ is the geometric Frobenius of $\Gal(\bar{k}/L)$ at $y$ and $\bar{y}$ is a geometric point over $y$. 

Fix $x\in E^{\times}$. Put
\[
    f_x:{\Gm}_E\longrightarrow{\Gm}_E,
        \quad T\longmapsto xT
\]
and 
\[
    \cF_x=f_x^*(\cK|_E).
\]
Recall that after fixing a geometric generic point $\bar{\eta}$, a lisse $\ell$-adic sheaf of finite rank over $k$ is equivalent to a finite dimensional continuous $\ell$-adic representation of $\pi_1({\Gm}_k,\bar{\eta})$ by taking fiber at $\bar{\eta}$. Let $V = (\cF_x)_{\bar{\eta}}$ be the corresponding $\ell$-adic representation of $\pi_1({\Gm}_k,\bar{\eta})$ and 
\[
W = \TInd_{E/k}(V)
\] 
be the tensor induction of $V$ from $\pi_1({\Gm}_E,\bar{\eta})$ to $\pi_1({\Gm}_k,\bar{\eta})$. Concretely, choose a lift $F \in \pi_1({\Gm}_k,\bar{\eta})$ of the geometric Frobenius of $\Gal(E/k)$. Take $1, F, F^{2}, \cdots, F^{n-1}$ as coset representatives of $\pi_1({\Gm}_E,\bar{\eta})$ in $\pi_1({\Gm}_k,\bar{\eta})$. As a representation of $\pi_1({\Gm}_E,\bar{\eta})$,
\begin{equation}\label{tensor representation}
    W = V \otimes V^F \otimes \cdots \otimes V^{F^{n-1}}
\end{equation}
where each $V^{F^j}$ is equal to $V$ as vector space with the action of $\sigma\in\pi_1({\Gm}_E,\bar{\eta})$ given by $F^{-j} \sigma F^j$ on $V$. The element $F$ acts on $W$ by
\begin{equation}\label{F permute}
    F(v_0\otimes v_1\otimes\cdots\otimes v_{n-1})=F^n(v_{n-1})\otimes v_0\otimes v_1\otimes\cdots\otimes v_{n-2}.
\end{equation}
This makes $W$ a representation of $\pi_1({\Gm}_k,\bar{\eta})$, which up to isomorphism does not depend on the choice of $F$. By \eqref{F permute} we have
\begin{equation}\label{trace of tensor product}
    \operatorname{tr}(F\mid W)=\operatorname{tr}(F^n\mid V)
\end{equation}
because the only contribution is given by elements like $v\otimes v\otimes \cdots \otimes v$. Let 
\[
\cG_x=\TInd_{E/k}(\cF_x)
\]
be the sheaf on ${\Gm}_k$ such that $(\cG_x)_{\bar{\eta}}=W$.
  
\begin{lemma}\label{tensor induction decomposition}
As a lisse $\ell$-adic sheaf on ${\Gm}_E$, one has
\begin{equation}\label{eq:Gx-geometric}
    \cG_x|_{E}\cong
        \bigotimes_{j=0}^{n-1}\cF_{x^{q^j}}.
\end{equation}
\end{lemma}

\begin{proof}
    Let $F_q$ be the $q$-th geometric Frobenius morphism on ${\Gm}_E$. Then for $0\le j\le n-1$, the sheaf corresponding to $V^{F^j}$ is exactly  $(F_{q}^j)^*\cF_x$. Moreover, we have
    \[
    (F_{q}^j)^*\cF_x\cong(f_x\circ F_{q}^j)^*\cK|_E\cong(F_{q}^j\circ f_{x^{q^{-j}}})^*\cK|_E\cong f_{x^{q^{-j}}}^*(F_{q}^j)^*\cK|_E\cong \cF_{x^{q^{-j}}}.
    \]
    The last isomorphism holds because $\cK$ is actually defined over ${\Gm}_k$. The result follows by \eqref{tensor representation}.
\end{proof}

\begin{lemma}\label{lem:Gx-basic}
The sheaf $\cG_x$ is lisse with rank $r^n$, pure of weight $n(r-1)$ and tame at $0$. Moreover, for every $t\in k^{\times}$, one has
\begin{equation}\label{eq:Gx-trace}
    \operatorname{tr}\left(\Frob_{k,t}\mid(\cG_x)_{\bar{t}}\right)=(-1)^{r-1} \Kl_{r,E,\psi}(tx).
\end{equation}
\end{lemma}

\begin{proof}
The assertions about rank, purity and tameness follow from \eqref{eq:Gx-geometric}. For the local trace, let $D_t$ (resp. $I_t$ ) be the decomposition group (resp. inertia group) at the closed point $t$. Choose a lift of $\Frob_{k,t}$ in $\pi_1({\Gm}_k,\bar{\eta})$, say $\widetilde{\Frob_{k,t}}$, via
\[
\Gal(\bar{k}/k)\cong D_t/I_t\twoheadleftarrow D_t\hookrightarrow\pi_1({\Gm}_k,\bar{\eta}).
\]
Then $\widetilde{\Frob_{k,t}}$ is also a lift of the geometric Frobenius of $\Gal(E/k)$ and by \eqref{trace of tensor product}, one has
\[
\operatorname{tr}\left(\Frob_{k,t}\mid(\cG_x)_{\bar{t}}\right)= \operatorname{tr}(\widetilde{\Frob_{k,t}}^n\mid V)=\operatorname{tr}\left(\Frob_{k,t}^n\mid (\cF_x)_{\bar{t}}\right)=\operatorname{tr}(\Frob_{E,tx}\mid \cK_{\bar{tx}}).
\]
The result follows by \eqref{eq:Kl-sheaf-trace}.
\end{proof}

Let $\phi_r$ be the morphism on ${\Gm}_k$ given by the $r$-th power $T\mapsto T^r$. For $a\in k^{\times}$, put
\begin{equation}\label{eq:H-def}
    \cH_{a,x}=\phi_r^*\cG_x\otimes\cL_{\psi_{-a}},
\end{equation}
where $\cL_{\psi_{-a}}$ is the Artin-Schreier sheaf associated with the additive character $t\mapsto\psi(-at)$. By Lemma~\ref{lem:Gx-basic}, the sheaf $\cH_{a,x}$ is lisse on ${\Gm}_k$ with rank $r^n$, pure of weight $n(r-1)$ and tame at $0$. Combining \eqref{eq:Gx-trace} and \eqref{eq:S-def} with Grothendieck-Lefschetz trace formula, we obtain
\begin{equation}\label{eq:GL trace formula}
S(E,r,x,a)= (-1)^{r-1}\sum_{t\in k^{\times}}\operatorname{tr}\left(\Frob_{k,t}\mid(\cH_{a,x})_{\bar{t}}\right)=\sum_{i=1}^2 (-1)^{i+r-1}\operatorname{tr}\left(\Frob_q\mid H_{c}^i({\Gm}_{\bar{k}},\cH_{a,x})\right).
\end{equation}
Note that $H_{c}^0$ vanishes because ${\Gm}_{\bar{k}}$ is not proper.

\section{The Swan conductor at infinity}\label{sec:swan}

We compute the Swan conductor of $\cH_{a,x}$ at $\infty$.

\begin{lemma}\label{lem:pure-insep}
Let $L/K$ be a finite purely inseparable extension of complete discretely valued fields of characteristic $p$ with perfect residue field. Fix a common algebraic closure and put $L^{\mathrm{sep}}=LK^{\mathrm{sep}}$. Then the restriction from $L^{\mathrm{sep}}$ to $K^{\mathrm{sep}}$ induces a canonical isomorphism
\[
\Gal(L^{\mathrm{sep}}/L)\xrightarrow{\sim}\Gal(K^{\mathrm{sep}}/K).
\]
Under this isomorphism, the lower numbering ramification filtrations and the upper numbering ramification filtrations are identified.  
\end{lemma}

\begin{proof}
Since $L/K$ is purely inseparable and $K^{\mathrm{sep}}/K$ is separable, one has
$L\cap K^{\mathrm{sep}}=K$, and $L$ is linearly disjoint with $K^{\mathrm{sep}}$ over $K$. Hence $LK^{\mathrm{sep}}$ is a separable closure
of $L$, and the restriction identifies the absolute Galois groups. By induction, we can assume $[L:K]=p$. Then the $p$-th power map is an isomorphism from $L$ to $K$ (here we use the assumption that the residue field is perfect) and is moreover an isometry. It uniquely extends to an isomorphism, still by $p$-th power, from $L^{\mathrm{sep}}=LK^{\mathrm{sep}}$ to $K^{\mathrm{sep}}$ and the extended map is still an isometry. The isomorphism between the absolute Galois group of $L$ and $K$ induced by the isometry coincides with the restriction, preserves the lower numbering ramification filtrations and hence the upper numbering ramification filtrations.
\end{proof}

\begin{proposition}\label{prop:swan}
For every $a\in k^{\times}$ and $x\in E^{\times}$, one has
\begin{equation}\label{eq:swan-formula}
    \Swan_\infty(\cH_{a,x})=r^n-m_x(a),
\end{equation}
where $m_x(a)$ is defined in \eqref{eq:mxa-def}.
\end{proposition}

\begin{proof}
Write $r=p^s r_0$ with $(r_0,p)=1$. Suppose first that $s\ge 1$. By \eqref{eq:Gx-geometric} all upper breaks of $\cG_x$ at $\infty$ are at most $\frac{1}{r}$. It is well known (for example, see \cite[1.13.1]{KatzGKM}) that upper breaks of $\phi_{r_0}^*\cG_x$ are just $r_0$ times upper breaks of $\cG_x$, and hence are at most $\frac{1}{p^s}$. By Lemma~\ref{lem:pure-insep}, pulling back along $\phi_p$ does not change the break. Thus, every upper break of $\phi_{r}^*\cG_x=\phi_{p^s}^*(\phi_{r_0}^*\cG_x)$ is at most $\frac{1}{p^s}$. The Artin-Schreier sheaf $\cL_{\psi_{-a}}$ has a single upper break $1$ at $\infty$. Therefore, every upper break of
$\cH_{a,x}=\phi_{r}^*\cG_x\otimes\cL_{\psi_{-a}}$ is dominated by $\cL_{\psi_{-a}}$ and is exactly $1$. Since the rank of $\cH_{a,x}$ is $r^n$, we have
\[
        \Swan_\infty(\cH_{a,x})=r^n.
\]
On the other hand, in this case $m_x(a)=0$ by \eqref{eq:m-zero-p-divides-r}. This proves \eqref{eq:swan-formula} when
$p\mid r$.

It remains to consider $s=0$, i.e., $p\nmid r$. Katz's local monodromy theorem for Kloosterman
sheaf at $\infty$ gives 
\begin{equation}\label{eq:Katz-local-infty}
    \phi_{r}^*\cK|_{P_\infty}\cong\bigoplus_{z^r=1}\cL_{\psi_{rz}},
\end{equation}
where $P_\infty$ is the wild inertial subgroup at $\infty$; see \cite[Remark~10.4.5]{KatzGKM}. Therefore, by \eqref{eq:Gx-geometric}
\[
\begin{aligned}
\cH_{a,x}|_{P_\infty}&\cong\cL_{\psi_{-a}}\otimes\phi_{r}^*\cG_x|_{P_\infty}\\
&\cong\cL_{\psi_{-a}}\otimes\bigotimes_{j=0}^{n-1}\phi_{r}^*\cF_{x^{q^j}}|_{P_\infty}\\
&\cong\cL_{\psi_{-a}}\otimes\bigotimes_{j=0}^{n-1}f_{x^{q^j/r}}^*(\phi_{r}^*\cK|_{P_\infty})\\
&\cong\cL_{\psi_{-a}}\otimes\bigotimes_{j=0}^{n-1}\bigoplus_{z_{j}^r=x^{q^j}}\cL_{\psi_{rz_j}}\\
&\cong\bigoplus_{z_j^r=x^{q^j},~0\le j\le n-1}\cL_{\psi_{r(z_0+\cdots+z_{n-1})-a}}.
\end{aligned}
\]
The result follows by the definition of $m_x(a)$ in \eqref{eq:mxa-def}.
\end{proof}

\section{Vanishing of \texorpdfstring{$H_c^2$}{Hc2}}\label{sec:H2}

To estimate \eqref{eq:GL trace formula}, we need
\[
        H_c^2({\Gm}_{\bar{k}},\cH_{a,x})=0.
\]
Fix a geometric generic point $\bar{\eta}$ and let $\pi_{1}^{\geom}=\pi_1({\Gm}_{\bar{k}},\bar{\eta})$. By Poincare duality, this is equivalent to proving the vanishing of
\begin{equation}\label{eq:geom coinv}
    H^0({\Gm}_{\bar{k}},\cH_{a,x}^{\vee})^{\vee}\cong(\cH_{a,x})_{\bar{\eta},\pi_{1}^{\geom}}.
\end{equation}
Here $(\cH_{a,x})_{\bar{\eta},\pi_{1}^{\geom}}$ is the $\pi_{1}^{\geom}$-coinvariant of $(\cH_{a,x})_{\bar{\eta}}$.   

\begin{lemma}\label{lem:H2-p-divides-r}
If $p\mid r$, then
\[
H_c^2({\Gm}_{\bar{k}},\cH_{a,x})=0,\quad H^0({\Gm}_{\bar{k}},\cH_{a,x})=0.
\]
\end{lemma}

\begin{proof}
It has been shown in the proof of Proposition~\ref{prop:swan} that every upper break of $\cH_{a,x}$ at infinity is exactly $1$. Thus, $(\cH_{a,x})_{\bar{\eta}}$ has zero $P_\infty$-coinvariants and $P_\infty$-invariants, which gives $H^0=0$ and $H_c^2=0$.
\end{proof}

For the rest of this section assume $p\nmid r$. For $c\in\bar{k}^\times$, put
\[
\cK_c=[T\mapsto cT]^*\cK|_{\bar{k}}.
\]

\begin{lemma}\label{lem:twists-duals}
(a) If $\cK_c\cong\cK_1$ up to a rank-$1$ twist, then $c=1$.\\
(b) If $\cK_c\cong\cK_{1}^\vee$ up to a rank-$1$ twist, then $c=(-1)^r$ and we have exactly $\cK_{(-1)^r}\cong\cK_{1}^\vee$.
\end{lemma}

\begin{proof}
If $\cK_c\cong\cK_1$ up to a rank-$1$ twist, then $\phi_r^*\cK_c\cong\phi_r^*\cK_1$ up to a rank-$1$ twist. The local decomposition \eqref{eq:Katz-local-infty} gives an equality of multisets
\begin{equation}\label{eq:A-equality}
        \{rz\mid z^r=c\}=\{rz\mid z^r=1\}+b
\end{equation}
for some $b\in\bar{k}$ coming from the rank-$1$ twist. Since $r\ge 2$, taking sums over the $r$ elements in \eqref{eq:A-equality} gives $0=0+rb$.
Since $p\nmid r$, we get $b=0$, which gives the first assertion. The second assertion follows by the first assertion and the isomorphism $\cK_{(-1)^r}\cong\cK_{1}^\vee$, which is a corollary of the rigidity of Kloosterman sheaf; see \cite[Corollary~4.1.4]{KatzGKM}.
\end{proof}

Now we recall Katz's global monodromy theorem of Kloosterman sheaf. Consider the geometric monodromy representation of $\cK$ at $\bar{\eta}$:
\[
\rho: \pi_1^{\geom}\longrightarrow \GL((\cK)_{\bar{\eta}})\cong\GL_r(\Ql).
\]
Let $G_{\geom}$ be the Zariski closure of $\rho(\pi_1^{\geom})$ in $\GL_r$ (as an algebraic group over $\Ql$).

\begin{theorem}\cite[Theorem 11.1]{KatzGKM}\label{global monodromy}
The geometric monodromy group $G_\geom$ is a connected algebraic group over $\Ql$ with simple Lie algebra. Explicitly,
\[
G_\geom=
\begin{cases}
\operatorname{Sp}_r, & r\text{ even},\\
\operatorname{SL}_r, & r\text{ odd and }p\text{ odd},\\
\operatorname{SO}_r, & p=2,\ r\text{ odd},\ r\ne 7,\\
G_2, & p=2,\ r=7.
\end{cases}
\]
\end{theorem}

We also need the following Goursat-Kolchin-Ribet criterion.
\begin{theorem}\cite[Proposition~1.8.2]{KatzESDE}\label{GKR}
Let $G$ be an algebraic group over some algebraically closed field $K$ with characteristic zero, $n\ge 2$ be an integer and $\rho_i:G\rightarrow\GL(V_i)~(1\le i\le n)$ be a set of finite dimensional irreducible representations of $G$ whose direct sum is faithful. Let $G_i=\rho_i(G)$ and suppose that
\begin{enumerate}[label=(\alph*)]
\item each $G_i^{\circ,\der}$ acts irreducibly on $V_i$ and $\Lie(G_i^{\circ,\der})$ is simple;
\item for $i\ne j$ the representations $(G_i^{\circ,\der},V_i)$ and $(G_j^{\circ,\der},V_j)$ are Goursat-adapted;
\item for $i\ne j$ the representations $\rho_i$ and $\rho_j$ are not isomorphic up to a twist;
\item for $i\ne j$ the representations $\rho_i^\vee$ and $\rho_j$ are not isomorphic up to a twist.
\end{enumerate}
Then $G^{\circ,\der}$ is the subgroup $\prod_{i=1}^n G_i^{\circ,\der}$ of $\prod_{i=1}^n\GL(V_i)$.
\end{theorem}

We will not give the definition of Goursat-adapted. Later we will point out that this is automatically satisfied in our case.

Return to our setting. Fix $x\in E^{\times}$. For $0\le i\le n-1$, put 
\[
\varphi_i: \pi_1^{\geom}\longrightarrow \GL_r(\Ql)
\]
be the geometric monodromy representation of $\cK_{x^{q^i}}$ and $G_{\geom,i}$ be the Zariski closure of $\varphi_i(\pi_1^{\geom})$. From a representation-theoretic viewpoint, pulling back along $T\mapsto cT$ for some $c\in\bar{k}$ corresponds to composing $\rho$ with an automorphism of $\pi_1^{\geom}$. Hence $G_{\geom,i}=G_{\geom}$ and we use the indices $i$ to distinguish different coordinates. For convenience, we still denote $\pi_1^{\geom}\rightarrow G_{\geom,i}$ by $\varphi_i$. 

We define an equivalence relation on $\{0,1,\cdots,n-1\}$. For $0\le i,j\le n-1$, set
\[
    i\sim j\quad\Longleftrightarrow\quad x^{q^i}=x^{q^j}\text{ or }x^{q^i}=(-1)^r x^{q^j}.
\]
Then by Lemma~\ref{lem:twists-duals}, if $i,j$ are in the same equivalence class, then $\varphi_i\cong\varphi_j$ or $\varphi_i\cong\varphi_j^\vee$ and otherwise $\varphi_i$ is neither isomorphic to $\varphi_j$ nor $\varphi_j^\vee$ even up to a twist. We choose one representative from each class. These representatives form a subset $R$ of $\{0,1,\cdots,n-1\}$. Combining with \eqref{eq:Gx-geometric}, the geometric monodromy representation of $\cG_x$ factors as
\begin{equation}\label{factor}
\pi_1^{\geom}\xrightarrow{\Delta_\varphi}\prod_{j\in R}G_{\geom,j}\longrightarrow\prod_{i=0}^{n-1}G_{\geom,i}\longrightarrow\GL((\Ql^r)^{\otimes n}),
\end{equation}
where $\Delta_\varphi(\sigma)=(\varphi_0(\sigma),\cdots,\varphi_{n-1}(\sigma))$ for $\sigma\in\pi_1^{\geom}$ and the second arrow is given diagonally on each equivalent class: if $i\sim j$, then the image of $g\in G_{\geom,j}$ on $G_{\geom,i}$ is $g$ if $\varphi_i\cong\varphi_j$ and $(g^t)^{-1}$ if $\varphi_i\cong\varphi_j^\vee$.

\begin{lemma}\label{zariski closure}
The Zariski closure $G$ of $\Delta_\varphi(\pi_1^{\geom})$ in $\prod_{j\in R}G_{\geom,j}$ is the whole $\prod_{j\in R}G_{\geom,j}$.
\end{lemma}
\begin{proof}
This is a direct corollary of Theorem~\ref{GKR}. For $i\in R$, let $p_i:\prod_{j\in R}G_{\geom,j}\rightarrow G_{\geom,i}$ be the projection onto the $i$-th coordinate and $\rho_i$ be the composition
\[
G\hookrightarrow\prod_{j\in R}G_{\geom,j}\xrightarrow{p_i}G_{\geom,i}\hookrightarrow\GL_r(\Ql^r).
\]
Then $G_i=\rho_i(G)$ is exactly $G_{\geom,i}$ because it is closed in $G_{\geom,i}$ and contains $\varphi_i(\pi_1^{\geom})$. Direct sum of all $\rho_i$ is faithful by the definition of $G$. By Theorem~\ref{global monodromy}, the Lie algebra of $G_{\geom,i}^{\circ,\der}=G_{\geom,i}$ is simple and the action of $G_{\geom,i}$ on $\Ql^r$ is the standard representation, which is irreducible. For $i\ne j$ the representations $(G_{\geom,i},\Ql^r)$ and $(G_{\geom,j},\Ql^r)$ are automatically Goursat-adapted by \cite[Examples 1.8.1]{KatzESDE}. Moreover, $\rho_i$ is not isomorphic to $\rho_j$ or $\rho_j^\vee$ up to a twist because $\varphi_i$ factors as $\rho_i\circ\Delta_\varphi$ and, as explained above, $\varphi_i$ is not isomorphic to $\varphi_j$ or $\varphi_j^\vee$ up to a twist. The result follows by Theorem~\ref{GKR}. 
\end{proof}

\begin{lemma}\label{lem:H2-vanish}
For every $r\ge 2$, every $a\in k^{\times}$, and every $x\in E^{\times}$, one has
\[
H_c^2({\Gm}_{\bar{k}},\cH_{a,x})=0,\quad H^0({\Gm}_{\bar{k}},\cH_{a,x})=0.
\]
\end{lemma}

\begin{proof}
The case $p\mid r$ is Lemma~\ref{lem:H2-p-divides-r}.  Assume $p\nmid r$. By \eqref{eq:geom coinv}, it is enough to show that $\phi_{r}^*\cG_x|_{\bar{k}}$ does not admit a rank-$1$ sub-quotient isomorphic to $\cL_{\psi_a}|_{\bar{k}}$. By \eqref{factor}, the geometric monodromy representation of $\phi_{r}^*\cG_x$ factors as 
\[
\pi_1^{\geom}\xrightarrow{\phi_{r,*}}\pi_1^{\geom}\xrightarrow{\Delta_\varphi}\prod_{j\in R}G_{\geom,j}\longrightarrow\GL((\Ql^r)^{\otimes n}).
\]
The map $\phi_r$ is a finite \'etale cover over $\Gm$ and thus the image of $\pi_1^{\geom}$ under $\phi_{r,*}$ is a subgroup of $\pi_1^{\geom}$ of finite index. Hence, by Lemma~\ref{zariski closure} and Theorem~\ref{global monodromy}, the Zariski closure of $\Delta_\varphi\circ\phi_{r,*}(\pi_1^{\geom})$ in $\prod_{j\in R}G_{\geom,j}$ contains $(\prod_{j\in R}G_{\geom,j})^{\circ}=\prod_{j\in R}G_{\geom,j}$, and hence must be equal to $\prod_{j\in R}G_{\geom,j}$. Therefore, any one-dimensional sub-quotient of the geometric monodromy representation of $\phi_{r}^*\cG_x$ is a character of $\prod_{j\in R}G_{\geom,j}$, which must be trivial. Since $a\neq 0$, the Artin-Schreier sheaf $\cL_{\psi_a}|_{\bar{k}}$ is geometrically nontrivial. Hence, it cannot be isomorphic to a rank-$1$ subquotient of $\phi_r^*\cG_x|_{\bar{k}}$. This gives the result.
\end{proof}

\section{Cohomological estimates}\label{sec:cohomological-estimates}

We now give the estimate for $S(E,r,x,a)$ and then deduce Theorem~\ref{main thm}.

\begin{lemma}\label{engenvalue 1}
The operator $\Frob_q$ has the eigenvalue $1$ on $H_c^1({\Gm}_{\bar{k}},\cH_{a,x})$.
\end{lemma}
\begin{proof}
This is essentially the last paragraph of \cite[Section 3]{LinWanEtale}. Let $j: \Gm \hookrightarrow \Pj^1$ be the standard open immersion. For any lisse $\ell$-adic sheaf $\cF$ on $\Gm$, the compactly supported cohomology $H^i_c({\Gm}_{\bar{k}}, \cF)$ is canonically isomorphic to $H^i(\Pj^1_{\bar{k}}, j_! \cF)$. Consider the short exact sequence of sheaves on $\Pj^1$:
\[
 0 \longrightarrow j_! \cH_{a,x} \longrightarrow j_* \cH_{a,x} \longrightarrow j_* \cH_{a,x} / j_! \cH_{a,x} \longrightarrow 0.
\]
Taking the long exact sequence, one has
\begin{equation}\label{long exact sequence}
     0 \longrightarrow H^0(\Pj^1_{\bar{k}}, j_! \cH_{a,x}) \longrightarrow H^0(\Pj^1_{\bar{k}}, j_* \cH_{a,x}) \longrightarrow H^0(\Pj^1_{\bar{k}}, j_* \cH_{a,x} / j_! \cH_{a,x}) \longrightarrow H^1(\Pj^1_{\bar{k}}, j_! \cH_{a,x}).
\end{equation}
By Lemma~\ref{lem:H2-vanish}
\[
H^0(\Pj^1_{\bar{k}}, j_* \cH_{a,x})=H^0({\Gm}_{\bar{k}}, \cH_{a,x})=0.
\]
The quotient sheaf $j_* \cH_{a,x} / j_! \cH_{a,x}$ is a skyscraper sheaf supported exactly at the boundary points $\{0, \infty\}$. Thus, we have 
\[
H^0(\Pj^1_{\bar{k}}, j_* \cH_{a,x} / j_! \cH_{a,x})=(\cH_{a,x})_{\bar{\eta}}^{I_0}\oplus(\cH_{a,x})_{\bar{\eta}}^{I_\infty},
\]
where $I_0$ (resp. $I_\infty$ ) is the inertial subgroup at $0$ (resp. $\infty$ ). Combining this with \eqref{long exact sequence}, we have a $\Frob_q$-equivariant injection
\begin{equation} \label{eq:injection}
(\cH_{a,x})_{\bar{\eta}}^{I_0}\oplus(\cH_{a,x})_{\bar{\eta}}^{I_\infty}\hookrightarrow H^1(\Pj^1_{\bar{k}}, j_! \cH_{a,x})\cong H^1_c({\Gm}_{\bar{k}}, \cH_{a,x}). 
\end{equation}
By \cite[Theorem 7.8]{SGA45} or \cite[Theorem 7.4.3]{KatzGKM}, $I_0$ acts unipotently on $\cK|_E$ with a single Jordan block and the geometric Frobenius $\Frob_{q^n}$ acts trivially on the rank-$1$ subspace $(\cK|_E)^{I_0}$. Multiplication by $x$ fixes the boundary point $0$, so the same assertion holds for $V=(\cF_x)_{\bar\eta}$. If $v\in V^{I_0}$ is a nonzero element, then 
\[
w:=v\otimes v\otimes\cdots\otimes v\in V \otimes V^F \otimes \cdots \otimes V^{F^{n-1}}=\TInd_{E/k}(V)
\]
is invariant under the action of $I_0$. Moreover, by \eqref{F permute} the action of $\Frob_q$ on $w$ is given by
\[
\Frob_q(w)=\Frob_{q^n}(v)\otimes v\otimes\cdots\otimes v=w.
\]
This gives a nonzero element of $(\cG_x)_{\bar{\eta}}^{I_0}$ with trivial $\Frob_q$-action. The same assertion holds for $\cH_{a,x}=\phi_r^*\cG_x\otimes\cL_{\psi_{-a}}$ because $\cL_{\psi_{-a}}$ is unramified at $0$. The result follows by \eqref{eq:injection}.
\end{proof}

\begin{proposition}\label{prop:S-swan-estimate}
For every $a\in k^{\times}$ and $x\in E^{\times}$, one has
\begin{equation}\label{eq:S-swan-estimate}
\left|S(E,r,x,a)-(-1)^r\right|\le \left(r^n-m_x(a)-1\right)q^{\frac{(r-1)n+1}{2}}.
\end{equation}
\end{proposition}

\begin{proof}
By \eqref{eq:GL trace formula} and Lemma~\ref{lem:H2-vanish}, we have 
\[
S(E,r,x,a)=(-1)^r\operatorname{tr}\left(\Frob_q\mid H_c^1({\Gm}_{\bar{k}},\cH_{a,x})\right).
\]
Recall that the sheaf $\cH_{a,x}$ is lisse on $\Gm$, pure of weight $n(r-1)$ and tame at $0$. The Grothendieck-Ogg-Shafarevich formula on
$\Gm=\Pj^1\setminus\{0,\infty\}$ and \eqref{eq:swan-formula} gives
\[
\dim H_c^1({\Gm}_{\bar{k}},\cH_{a,x})=-\chi_c({\Gm}_{\bar{k}},\cH_{a,x})=\Swan_0(\cH_{a,x})+\Swan_\infty(\cH_{a,x})=r^n-m_x(a).
\]
Deligne's theorem gives that $H_c^1({\Gm}_{\bar{k}},\cH_{a,x})$ is mixed of weight $\le n(r-1)+1$. By Lemma~\ref{engenvalue 1}, the space $H_c^1({\Gm}_{\bar{k}},\cH_{a,x})$ admits $1$ as a $\Frob_q$-eigenvalue. Subtracting this eigenvalue and applying the triangle inequality give \eqref{eq:S-swan-estimate} and hence \eqref{eq:main}. This finishes the proof of Theorem~\ref{main thm}.
\end{proof}

\section{Counterexample}\label{counterexample}
The original estimate in \cite[Conjecture 4.3]{WanNormTrace} (in the field extension case) is
\begin{equation}\label{Wan-conj}
\left|\Nm(E,r,x,a)-\frac{(q^n-1)^{r-1}}{q}\right|
\le r q^{\frac{(r-1)n-1}{2}}.
\end{equation}

We present a counterexample to \eqref{Wan-conj}. Consider the quadratic extension $E = \mathbb{F}_{121} = \mathbb{F}_{11}(i)$ over the base field $k = \mathbb{F}_{11}$ with $i^2 = -1$. Take $r=2,~a = 3 \in \mathbb{F}_{11}^{\times}$ and $x = 2 + 5i \in \mathbb{F}_{121}^{\times}$. Then by \eqref{nbrxa}, the count is given exactly by
\[
\Nm_{2,2}(3, 2+5i) = \#\left\{ g \in E^{\times} \mid \Tr_{E/k}\left(g + \frac{2+5i}{g}\right) = 3 \right\}.
\]
Let $z = g + \frac{x}{g} = \alpha + \beta i \in E$ with $\alpha,\beta\in k$. Assume $\Tr_{E/k}(z) = 2\alpha = 3$. Then $\alpha =7$. For a fixed $z$, the number of roots of the equation $g^2 - zg + x = 0$ in $E^{\times}$ is $1 + \chi_E(z^2 - 4x)$, where $\chi_E$ denotes the quadratic character on $E$. Using $\chi_E(\cdot) = \chi_{11}\circ\Nr_{E/K}(\cdot)$, the total number of solutions is 
\[
\begin{aligned}
\Nm_{2,2}(3,2+5i) &= \sum_{\beta \in \mathbb{F}_{11}} \left[ 1 + \chi_E\left((7+\beta i)^2 - 4(2+5i)\right) \right]\\ 
&= 11 + \sum_{\beta \in \mathbb{F}_{11}} \chi_{11}(\beta^4 + 4\beta^2 + \beta + 2)=4.
\end{aligned}
\]

We claim $m_x(3)=0$. Otherwise, there are $z_0,z_1\in\bar{k}$ such that
$z_0^2=x,~z_1^2=x^{11}$ and 
$2(z_0+z_1)=3.$ 
Thus, we have $z_0=9+7i$, which implies that 
\[
z_0^2=10+5i\neq 2+5i=x.
\]
This is a contradiction. 

Subtracting the main term $\frac{|E^\times|}{q} = \frac{120}{11}$ gives an error $\left| 4 - \frac{120}{11} \right| = \frac{76}{11} \approx 6.909$. However, the error bound in \eqref{Wan-conj} is $2 \sqrt{11} \approx 6.633$. Hence, the estimate \eqref{eq:main} holds, while  \eqref{Wan-conj} does not. 

\section{Construction of the semisimple Kloosterman sheaf}
\label{sec:ss-regular-sheaf}

The final two sections are devoted to generalizing Theorem~\ref{main thm} to the semisimple case. The purpose of this section is to construct a lisse sheaf whose local trace gives the Fourier transform of the semisimple product-trace count. Throughout the final two sections, let $k=\F_q$ and $B$ be as in
\eqref{B}. Put $E_i=\F_{q^{n_i}}$ and let $\Nm(B,r,x,a)$ be as in
\eqref{nbrxa}.  Put
\[
 N=\dim_k B=\sum_{i=1}^{s}n_i d_{i}^2,
 \qquad M=\sum_{i=1}^{s}n_i d_i.
\]

For any finite field $F$, 
recall that an element $y\in\GL_d(F)$ is called {regular} if its minimal
polynomial has degree $d$, or equivalently, if its minimal polynomial is equal to its
characteristic polynomial. It is called regular elliptic if its characteristic polynomial is irreducible over $F$. An element
$x=(x_i)_i\in B^{\times}$ is regular if every $x_i$ is regular. This is also the
notion used in \cite{WanNormTrace} and \cite{ZelingherMatrix}.

For $y\in\GL_a(F)$ and $b\ge 1$, let
\begin{equation}\label{eq:ss-J-block}
J_b(y)=
\begin{pmatrix}
y&I_a&0&\cdots&0\\
0&y&I_a&\ddots&\vdots\\
\vdots&\ddots&\ddots&\ddots&0\\
0&\cdots&0&y&I_a\\
0&\cdots&\cdots&0&y
\end{pmatrix}\in\GL_{ab}(F).
\end{equation}

The following is a standard consequence of the rational canonical form for regular matrices over a finite field. We record it only to fix the notations for the primary data used below.

\begin{proposition}\label{prop:ss-primary-decomposition}
Let $F$ be a finite field, and let $y\in\GL_d(F)$ be regular. Then there exist pairwise distinct monic irreducible polynomials
$f_1,\cdots,f_t\in F[T]$, positive integers $b_1,\cdots,b_t$, and regular
elliptic elements $y_j\in\GL_{a_j}(F)$ with characteristic polynomial
$f_j$, such that
\begin{equation}\label{eq:ss-primary-decomposition} y\sim\operatorname{diag}\left(J_{b_1}(y_1),\cdots,J_{b_t}(y_t)\right), \qquad \sum_{j=1}^{t}a_j b_j=d.
\end{equation}
The unordered collection of pairs $(f_j,b_j)$ is uniquely determined by $y$.
\end{proposition}

Fix a regular element $x=(x_i)\in B^{\times}$. For each $i$, assume that
\begin{equation}\label{eq:ss-x-primary}
x_i\sim\operatorname{diag}\left(
J_{b_{i1}}(y_{i1}),\cdots,J_{b_{i t_i}}(y_{i t_i})\right), \qquad \sum_{j=1}^{t_i}a_{ij}b_{ij}=d_i,
\end{equation}
where $y_{ij}\in\GL_{a_{ij}}(E_i)$ is regular elliptic. Choose one
eigenvalue 
\begin{equation}\label{eq:ss-xi}
\xi_{ij}\in\F_{q^{n_i a_{ij}}}^{\times}
\end{equation}
of $y_{ij}$ from the $E_i$-Frobenius orbit, and put
\begin{equation}\label{eq:ss-lambda-data}
\Lambda=\{(i,j)\mid 1\le i\le s,~1\le j\le t_i\},
\qquad f_{ij}=n_i a_{ij}.
\end{equation}
Define
\begin{equation}\label{eq:ss-rank}
R_{B,x}=\prod_{(i,j)\in\Lambda}\binom{b_{ij}+r-1}{b_{ij}}^{f_{ij}}.
\end{equation}
and
\begin{equation}\label{eq:ss-prefactors}
A_B=\frac{(r-1)(N-M)}{2}, \qquad \varepsilon_B=(-1)^{(r-1)\sum_{i=1}^{s}d_i}.
\end{equation}
The integer $A_B$ is integral because
$N-M=\sum_i n_i d_i(d_i-1)$ is even.

In the remainder of this section, we construct the sheaf needed to prove
Theorem~\ref{thm:ss-main}.  Its local monodromy and cohomology will be analyzed in Section~\ref{sec:ss-cohomology-proof}.

Fix a non-trivial additive character
\[
\psi:k\longrightarrow \Ql^{\times}.
\]
For $y=(y_i)_i\in B^{\times}$, define the product-trace Kloosterman sum over $B$ by
\begin{equation}\label{eq:ss-KB}
K(B,r,y,\psi)=\sum_{\substack{g_1,\cdots,g_r\in B^{\times}\\g_1\cdots g_r=y}}\psi\left(\Tr_{B/k}(g_1+\cdots+g_r)\right).
\end{equation}
This is a generalization of the hyper-Kloosterman sum in \eqref{eq:Kl-def}. As both the product condition and the reduced trace are componentwise, we have 
\begin{equation}\label{eq:ss-factorization}
K(B,r,y,\psi)=\prod_{i=1}^{s}K(M_{d_i}(E_i),r,y_i,\psi)=\prod_{i=1}^{s}K(M_{d_i}(E_i),r,y_i,\psi_i),
\end{equation}
where $\psi_i=\psi\circ\Tr_{E_i/k}$ and in $K(M_{d_i}(E_i),r,y_i,\psi_i)$ we view $M_{d_i}(E_i)$ as an algebra over $E_i$.

\begin{lemma}\label{lem:ss-fourier}
For $a\in k^{\times}$ and $x=(x_i)_i\in B^{\times}$, one has
\begin{equation}\label{eq:ss-fourier}
 \Nm(B,r,x,a)=\frac{|B^{\times}|^{r-1}}{q}+\frac{1}{q}S(B,r,x,a),
\end{equation}
where
\begin{equation}\label{eq:ss-SB}
 S(B,r,x,a)=\sum_{t\in k^{\times}}\psi(-at)K(B,r,t^r x,\psi).
\end{equation}
\end{lemma}

\begin{proof}
The proof is similar to that of Lemma~\ref{lem:fourier}. Additive orthogonality applied to the condition in \eqref{nbrxa}
gives a sum over $t\in k$. The term $t=0$ is
$|B^{\times}|^{r-1}/q$. For $t\neq 0$, replacing each $g_j$ by $tg_j$, with
$t$ acting as a central scalar in every simple factor, changes the product
condition to $g_1\cdots g_r=t^r x$ and gives \eqref{eq:ss-SB}. 
\end{proof}

Now we introduce Zelingher's formula. Let $F=\F_{q'}$ be a finite extension of $\F_p$ and $\psi_F$ a non-trivial additive character of $F$.

\begin{theorem}[Zelingher]\label{thm:ss-zelingher}
Let $y\in\GL_d(F)$ be regular, and write
\[
y\sim\operatorname{diag}\left(J_{b_1}(y_1),\cdots,J_{b_t}(y_t)\right)
\]
as in Proposition~\ref{prop:ss-primary-decomposition}, where
$y_j\in\GL_{a_j}(F)$ is regular elliptic. Let $F_j=\F_{{q'}^{a_j}}$ and $\psi_{F_j}=\psi_F\circ\Tr_{F_j/F}$. Choose an eigenvalue $\eta_j\in F_j^\times$ of $y_j$. Denote by
$\cK_{r,\psi_{F_j}}$ the rank-$r$ Kloosterman sheaf over $F_j$ normalized as in \eqref{eq:Kl-sheaf-trace}. Then
\begin{equation}\label{eq:ss-zelingher}
K(M_d(F),r,y,\psi_F)=(-1)^{(r-1)d}{q'}^{(r-1)\binom d2}\prod_{j=1}^{t}
\operatorname{tr}\left(\Frob_{F_j,\eta_j}\mid
(\Sym^{b_j}\cK_{r,\psi_{F_j}})_{\bar{\eta}_j}\right).
\end{equation}
\end{theorem}

\begin{proof}
This is \cite[Corollary~4.7]{ZelingherMatrix}, specialized to trivial multiplicative characters. 
\end{proof}

Recall that $\cK=\cK_{r,\psi}$ is the rank $r$ Kloosterman sheaf on ${\Gm}_k$. For $\lambda=(i,j)\in\Lambda$, abbreviate
\begin{equation}\label{eq:ss-abbreviations}
f_{\lambda}=n_i a_{ij},\qquad
b_{\lambda}=b_{ij},\qquad 
\xi_{\lambda}=\xi_{ij},\qquad
L_{\lambda}=\F_{q^{f_{\lambda}}}.
\end{equation}
Let $[\xi_{\lambda}]:{\Gm}_{L_{\lambda}}\to{\Gm}_{L_{\lambda}}$ denote multiplication
by $\xi_{\lambda}$, and put
\begin{equation}\label{eq:ss-factor-sheaf}
\cF_{\lambda}=[\xi_{\lambda}]^*
 \Sym^{b_{\lambda}}(\cK|_{L_{\lambda}}),
\end{equation}
\begin{equation}\label{eq:ss-M-sheaf}
\mathcal{M}_{B,x}=\bigotimes_{\lambda\in\Lambda}\TInd_{L_{\lambda}/k}(\cF_{\lambda})
 \qquad\text{on }{\Gm}_k.
\end{equation}
For the construction of the tensor induction, see Section~\ref{sec:tensor-induction}. The same argument as in Lemma~\ref{tensor induction decomposition} gives
\begin{equation}\label{eq:ss-geometric-decomposition}
\mathcal{M}_{B,x}|_{\bar{k}}\cong
\bigotimes_{\lambda\in\Lambda}\bigotimes_{h=0}^{f_{\lambda}-1}\Sym^{b_{\lambda}}\cK_{\xi_{\lambda}^{q^h}}.
\end{equation}
Here $\cK_c=[T\mapsto cT]^*\cK|_{\bar{k}}$ for $c\in \bar{k}$.

\begin{proposition}\label{prop:ss-M-properties}
The sheaf $\mathcal{M}_{B,x}$ is lisse of rank $R_{B,x}$, pure of weight
$(r-1)M$, and tame at $0$. Moreover, for every $t\in k^{\times}$, one has
\begin{equation}\label{eq:ss-K-trace}
K(B,r,tx,\psi)=\varepsilon_Bq^{A_B}
\operatorname{tr}\left(\Frob_{k,t}\mid(\mathcal{M}_{B,x})_{\bar{t}}\right).
\end{equation}
\end{proposition}

\begin{proof}
The rank of $\Sym^{b_{\lambda}}\cK$ is
$\binom{b_{\lambda}+r-1}{b_{\lambda}}$, so the
definition \eqref{eq:ss-M-sheaf} gives the rank
\eqref{eq:ss-rank}. Since $\cK$ is pure of weight $r-1$, the factor in \eqref{eq:ss-M-sheaf} indexed
by $\lambda$ has weight $f_{\lambda} b_{\lambda}(r-1)$, and the total weight of $\mathcal{M}_{B,x}$ is
therefore
\[
 (r-1)\sum_{i,j}n_i a_{ij}b_{ij}=(r-1)\sum_{i=1}^s n_i d_i = (r-1)M.
\]
Tameness at $0$ follows from the corresponding property of $\cK$.

For $t\in k^{\times}$, we have
\begin{align*}
\operatorname{tr}(\Frob_{k,t}\mid(\mathcal{M}_{B,x})_{\bar{t}}) &= \prod_{\lambda\in\Lambda}
\operatorname{tr}(\Frob_{k,t}\mid(\TInd_{L_{\lambda}/k}(\cF_{\lambda}))_{\bar{t}})\\
&=\prod_{\lambda\in\Lambda}
\operatorname{tr}(\Frob_{k,t}^{f_{\lambda}}\mid(\cF_{\lambda})_{\bar{t}})\\
&=\prod_{\lambda\in\Lambda}\operatorname{tr}(\Frob_{L_{\lambda},t\xi_{\lambda}}\mid(\Sym^{b_{\lambda}}(\cK|_{L_{\lambda}}))_{\bar{t\xi_{\lambda}}})\\
&=\prod_{i=1}^s (-1)^{(r-1)d_i}q^{-(r-1)n_i \binom{d_i}{2}} K(M_{d_i}(E_i),r,tx_i,\psi_i)\\
&= \varepsilon_{B}^{-1} q^{-A_B}K(B,r,tx,\psi).
\end{align*}
Here the first equality is due to \eqref{eq:ss-M-sheaf}. The second equality is due to the property of tensor induction; see \eqref{trace of tensor product} and the proof of Lemma~\ref{lem:Gx-basic}. The third equality is due to \eqref{eq:ss-factor-sheaf}. The fourth equality is due to Theorem~\ref{thm:ss-zelingher}, taking $F=E_i$, $\psi_F=\psi_i$, $y=tx_i$ and $F_j=L_{\lambda}$. Note that the matrix $tJ_{b_{\lambda}}(y_{\lambda})$ is conjugate to $J_{b_{\lambda}}(ty_{\lambda})$ by $\operatorname{diag}(I_{a_{\lambda}},t^{-1}I_{a_{\lambda}},\cdots,t^{-(b_{\lambda}-1)}I_{a_{\lambda}})$ and therefore $t x_i$ has the
same primary block sizes as $x_i$ and the eigenvalue associated with $t x_i$ in \eqref{eq:ss-xi} can be taken to be $t \xi_{\lambda}$.
The last equality is due to \eqref{eq:ss-prefactors} and \eqref{eq:ss-factorization}. This proves \eqref{eq:ss-K-trace}.
\end{proof}

Finally, for $a\in k^{\times}$, define
\begin{equation}\label{eq:ss-H-sheaf}
\cH_{B,x,a}=\phi_r^*\mathcal{M}_{B,x}\otimes\cL_{\psi_{-a}},
\end{equation}
where $\phi_r$ is the $r$-th power morphism on ${\Gm}_k$. By Proposition~\ref{prop:ss-M-properties} the sheaf $\cH_{B,x,a}$ is lisse on ${\Gm}_k$ of rank $R_{B,x}$, pure of weight $(r-1)M$ and tame at $0$. Combining \eqref{eq:ss-SB} and \eqref{eq:ss-K-trace} with Grothendieck-Lefschetz trace formula, we obtain
\begin{equation}\label{eq:ss-sum-as-traces}
S(B,r,x,a)=\varepsilon_Bq^{A_B}\sum_{t\in k^{\times}}\operatorname{tr}\left(\Frob_{k,t}\mid(\cH_{B,x,a})_{\bar{t}}\right)=\varepsilon_Bq^{A_B}\sum_{i=1}^2 (-1)^i \operatorname{tr}\left(\Frob_q\mid H_c^i({\Gm}_{\bar{k}},\cH_{B,x,a})\right).
\end{equation}

\section{Cohomology and proof of the semisimple estimate}
\label{sec:ss-cohomology-proof}

We now analyze the sheaf $\cH_{B,x,a}$ constructed in
\eqref{eq:ss-H-sheaf}. The proof proceeds in four steps. We first compute its Swan conductor at infinity. We then use product monodromy to prove the
vanishing of $H_c^2$. Next, we isolate a Frobenius eigenvalue equal to $1$ coming from the boundary point $0$. Finally, Deligne's theorem and the
Grothendieck-Ogg-Shafarevich formula give Theorem~\ref{thm:ss-main}. We use the notation in Section~\ref{sec:ss-regular-sheaf}.

We define the cancellation multiplicity at infinity. If $p\mid r$, put
\begin{equation}\label{eq:ss-m-p-divides-r}
 m_{B,x}(a)=0.
\end{equation}
Assume $p\nmid r$. For every $\lambda\in\Lambda$ and
$0\leq h<f_{\lambda}$, consider nonnegative integers
\begin{equation}\label{eq:ss-m-family}
m_{\lambda,h,z}\ge 0\qquad (z^r=\xi_{\lambda}^{q^h})
\end{equation}
subject to
\begin{equation}\label{eq:ss-weak-composition}
\sum_{z^r=\xi_{\lambda}^{q^h}}m_{\lambda,h,z}=b_{\lambda}.
\end{equation}
Then $m_{B,x}(a)$ is the number of all such families satisfying
\begin{equation}\label{eq:ss-cancellation}
r\sum_{\lambda\in\Lambda}\sum_{h=0}^{f_{\lambda}-1}
 \sum_{z^r=\xi_{\lambda}^{q^h}}m_{\lambda,h,z}z=a.
\end{equation}
The definition is independent of the choices of the representatives $\xi_{\lambda}$ in \eqref{eq:ss-xi} because replacing $\xi_{\lambda}$ by $\xi_{\lambda}^{q^u} ~(1\le u\le a_{\lambda}$) cyclically permutes the multiset $\{\xi_{\lambda}^{q^h}\mid0\le h<f_{\lambda}\}$.

\begin{proposition}\label{prop:ss-swan}
For regular $x\in B^{\times}$ and $a\in k^{\times}$, one has
\begin{equation}\label{eq:ss-swan-formula}
\Swan_\infty(\cH_{B,x,a})=R_{B,x}-m_{B,x}(a).
\end{equation}
\end{proposition}

\begin{proof}
First, assume that $p\mid r$. The same argument as in Proposition~\ref{prop:swan} shows that every upper break of $H_{B,x,a}$ is dominated by the Artin-Schreier sheaf $\cL_{\psi_{-a}}$ and is exactly $1$. In this case \eqref{eq:ss-swan-formula} follows from \eqref{eq:ss-m-p-divides-r}.

Now assume that $p\nmid r$. By \eqref{eq:Katz-local-infty}, \eqref{eq:ss-H-sheaf} and \eqref{eq:ss-geometric-decomposition} we have
\begin{align*}
H_{B,x,a}|_{P_\infty}&\cong\cL_{\psi_{-a}}\otimes
\bigotimes_{\lambda\in\Lambda}\bigotimes_{h=0}^{f_{\lambda}-1}
\phi_r^*(\Sym^{b_{\lambda}}\cK_{\xi_{\lambda}^{q^h}})|_{P_\infty}\\
&\cong\cL_{\psi_{-a}}\otimes\bigotimes_{\lambda\in\Lambda}
\bigotimes_{h=0}^{f_{\lambda}-1}
\Sym^{b_{\lambda}}\left(\bigoplus_{z^r=\xi_{\lambda}^{q^h}}\cL_{\psi_{rz}}\right)\\
&\cong\cL_{\psi_{-a}}\otimes
\bigotimes_{\lambda\in\Lambda}
\bigotimes_{h=0}^{f_{\lambda}-1}\bigoplus_{\substack{m_z\ge 0,~z^r=\xi_{\lambda}^{q^h}\\
\sum_z m_z=b_{\lambda}}}
\cL_{\psi_{r\sum_z m_z z}}.
\end{align*}
The formula above contains exactly $R_{B,x}$ rank-$1$ summands, and each summand indexed by a family
\eqref{eq:ss-m-family} has Artin-Schreier parameter
\[
r\sum_{\lambda,h,z}m_{\lambda,h,z}z-a.
\]
It is tame precisely when this parameter vanishes. The result follows by the definition of $m_{B,x}(a)$.
\end{proof}

\begin{proposition}\label{prop:ss-vanishing}
For regular $x\in B^{\times}$ and $a\in k^{\times}$, one has
\[
H^0({\Gm}_{\bar{k}},\cH_{B,x,a})=0, \quad H_c^2({\Gm}_{\bar{k}},\cH_{B,x,a})=0.
\]
\end{proposition}
\begin{proof}
If $p\mid r$, the proof of Proposition~\ref{prop:ss-swan} shows that every upper break of
$\cH_{B,x,a}$ at infinity is $1$. Thus, both the $P_\infty$-invariant and the $P_\infty$-coinvariant of the geometric representation are zero, which gives $H^0=0$ and $H_c^2=0$.

Now assume $p\nmid r$. It suffices to show that every rank-$1$ sub-quotient of $\phi_r^*\mathcal{M}_{B,x}$ is geometrically constant. The argument is the same as in Section~\ref{sec:H2}. The geometric monodromy representation of $\phi_r^*\mathcal{M}_{B,x}$ factors through a product of certain copies of the geometric monodromy group $G_\geom$ appearing in Theorem~\ref{global monodromy}, which by Theorem~\ref{GKR} is the Zariski closure of the image of $\pi_1^{\geom}$; see Lemma~\ref{zariski closure} and the discussion above it. Therefore, any one-dimensional sub-quotient of the geometric monodromy representation of $\phi_r^*\mathcal{M}_{B,x}$ is a character of this product of certain copies of $G_\geom$, which must be trivial.
\end{proof}

\begin{proposition}\label{prop:ss-boundary-eigenvalue}
The operator $\Frob_q$ has the eigenvalue $1$ on
$H_c^1({\Gm}_{\bar{k}},\cH_{B,x,a})$.
\end{proposition}

\begin{proof}
By the exact sequence \eqref{long exact sequence} (change $\cH_{a,x}$ to $\cH_{B,x,a}$) and Proposition~\ref{prop:ss-vanishing}, we have a $\Frob_q$-equivariant injection
\[
(\cH_{B,x,a})_{\bar{\eta}}^{I_0}\oplus(\cH_{B,x,a})_{\bar{\eta}}^{I_\infty}
\hookrightarrow H_c^1({\Gm}_{\bar{k}},\cH_{B,x,a}).
\]
We have explained in Lemma~\ref{engenvalue 1} that for each $\lambda\in\Lambda$, the geometric Frobenius $\Frob_{L_{\lambda}}$ acts trivially on the rank-$1$ subspace $\left(\cK|_{L_{\lambda}}\right)_{\bar{\eta}}^{I_0}$. If $v_{\lambda}$ spans $\left(\cK|_{L_{\lambda}}\right)_{\bar{\eta}}^{I_0}$,
then $v_{\lambda}^{b_{\lambda}}$ is a nonzero $I_0$-invariant vector in
$\Sym^{b_{\lambda}}(\cK|_{L_{\lambda}})$ with trivial $\Frob_{L_{\lambda}}$-action. Multiplication by $\xi_{\lambda}$ fixes the boundary point $0$, so such a vector occurs in $\cF_{\lambda}=[\xi_{\lambda}]^*
\Sym^{b_{\lambda}}(\cK|_{L_{\lambda}})$. The same argument as Lemma~\ref{engenvalue 1} gives a nonzero element of $\left(\TInd_{L_{\lambda}/k}(\cF_{\lambda})\right)^{I_0}$ with trivial $\Frob_q$-action. The same assertion holds for $\mathcal{M}_{B,x}=
\bigotimes_{\lambda\in\Lambda}\TInd_{L_{\lambda}/k}(\cF_{\lambda})$ and therefore for $\cH_{B,x,a}=\phi_r^*\mathcal{M}_{B,x}\otimes\cL_{\psi_{-a}}$. This gives the result.
\end{proof}

\begin{proposition}\label{prop:ss-trace-estimate}
For every $a\in k^{\times}$ and $x\in B^{\times}$, one has
\begin{equation}\label{eq:ss-trace-estimate}
\left|S(B,r,x,a)+\varepsilon_Bq^{A_B}\right|
\leq\left(R_{B,x}-m_{B,x}(a)-1\right)
 q^{\frac{(r-1)N+1}{2}}.
\end{equation}
\end{proposition}

\begin{proof}
By \eqref{eq:ss-sum-as-traces} and Proposition~\ref{prop:ss-vanishing}, we have
\begin{equation}\label{eq:ss-cohomological-trace}
S(B,r,x,a)=-\varepsilon_Bq^{A_B}
\operatorname{tr}\left(\Frob_q\mid H_c^1({\Gm}_{\bar{k}},\cH_{B,x,a})\right).
\end{equation}
Since $\cH_{B,x,a}$ is tame at $0$, the
Grothendieck-Ogg-Shafarevich formula on
$\Gm=\Pj^1\setminus\{0,\infty\}$, together with
Propositions~\ref{prop:ss-swan} and \ref{prop:ss-vanishing}, gives
\begin{equation}\label{eq:ss-H1-dimension}
\dim H_c^1({\Gm}_{\bar{k}},\cH_{B,x,a})=\Swan_0(\cH_{B,x,a})+\Swan_\infty(\cH_{B,x,a})=R_{B,x}-m_{B,x}(a).
\end{equation}
The sheaf $\cH_{B,x,a}$ is pure of weight $(r-1)M$, so Deligne's theorem gives that $H_c^1({\Gm}_{\bar{k}},\cH_{B,x,a})$ is mixed of weight $\le (r-1)M+1$. Proposition~\ref{prop:ss-boundary-eigenvalue} provides
one $\Frob_q$-eigenvalue equal to $1$. Subtracting this eigenvalue from the trace in
\eqref{eq:ss-cohomological-trace} and applying the triangle inequality give
\eqref{eq:ss-trace-estimate} and hence \eqref{eq:ss-main}. This finishes the proof of Theorem~\ref{thm:ss-main}.
\end{proof}

\begin{remark}\label{rem:ss-field-specialization}
If $B=E=\F_{q^n}$, then $s=d_1=1,~n_1=n$, and $N=M=n$. The unique primary block in \eqref{eq:ss-J-block} has $a_{11}=b_{11}=1$ and $\xi_{11}=x$. Consequently,
\[
R_{B,x}=r^n,\qquad m_{B,x}(a)=m_x(a),\qquad
A_B=0,\qquad \varepsilon_B=(-1)^{r-1}.
\]
Thus, \eqref{eq:ss-main} specializes exactly to \eqref{eq:main}.
\end{remark}

\section*{Acknowledgements}

The authors are very grateful to Professor Daqing Wan for introducing this problem at a conference held at Nanjing University in April 2026 and for encouraging the authors to investigate it. The third author would like to thank Professor Weizhe Zheng and Xiangdong Wu for helpful discussions. The first author is supported by NSFC 12231009.

\section*{AI Use Disclosure}
ChatGPT-5.4 was used for exploratory computations, including the Swan conductor at infinity and the coefficients in the formula \eqref{eq:ss-main}. All content generated by AI has been completely rewritten and independently checked by the authors. The authors take full responsibility for their validity.


\begin{thebibliography}{99}

\bibitem{SGA45}
P. Deligne and J.-F. Boutot, Cohomologie \'etale: les points de d\'epart, in {\it Cohomologie \'etale}, 4--75, Lecture Notes in Math., 569, Springer, Berlin, 1977.

\bibitem{EichlerGauss}
M. Eichler, Allgemeine Kongruenzklasseneinteilungen der Ideale einfacher Algebren \"uber algebraischen Zahlk\"orper und ihre $L$-Reihen, J. Reine Angew. Math. {\bf 179} (1938), 227--251.

\bibitem{KatzGKM}
N.~M. Katz, Gauss sums, Kloosterman sums, and monodromy groups, Annals of Mathematics Studies, 116, Princeton Univ. Press, Princeton, NJ, 1988.

\bibitem{KatzESDE}
N.~M. Katz, Exponential sums and differential equations, Annals of Mathematics Studies, 124, Princeton Univ. Press, Princeton, NJ, 1990.

\bibitem{KatzSoto}
N.~M. Katz, Estimates for Soto-Andrade sums, J. Reine Angew. Math. {\bf 438} (1993), 143--161.


\bibitem{KimGauss}
D.~S. Kim, Gauss sums for general and special linear groups over a finite field, Arch. Math. (Basel) {\bf 69} (1997), no.~4, 297--304.

\bibitem{LamprechtGauss}
E. Lamprecht, Struktur und Relationen allgemeiner Gausscher Summen in endlichen Ringen. I, II, J. Reine Angew. Math. {\bf 197} (1957), 1--26, 27--48.

\bibitem{LiHuGauss}
Y. Li and S. Hu, Gauss sums over some matrix groups, J. Number Theory {\bf 132} (2012), no.~12, 2967--2976.

\bibitem{LinWanEtale}
X. Lin and D. Wan, Counting elements with given trace and norm in \'etale algebras, Int. J. Number Theory {\bf 21} (2025), no.~8, 1955--1965.


\bibitem{MoisioKloosterman}
M.~J. Moisio, Kloosterman sums, elliptic curves, and irreducible polynomials with prescribed trace and norm, Acta Arith. {\bf 132} (2008), no.~4, 329--350.

\bibitem{MoisioWan}
M.~J. Moisio and D. Wan, On Katz's bound for the number of elements with given trace and norm, J. Reine Angew. Math. {\bf 638} (2010), 69--74.

\bibitem{RojasLeonTraceNorm}
A. Rojas-Le\'on, Rationality of trace and norm $L$-functions, Duke Math. J. {\bf 161} (2012), no.~9, 1751--1795.


\bibitem{RojasLeonWanMoments}
A. Rojas-Le\'on and D. Wan, Moment zeta functions for toric Calabi-Yau hypersurfaces, Commun. Number Theory Phys. {\bf 1} (2007), no.~3, 539--578.

\bibitem{WanLectures}
D. Wan, Lectures on zeta functions over finite fields, in {\it Higher-dimensional geometry over finite fields}, 244--268, NATO Sci. Peace Secur. Ser. D Inf. Commun. Secur., 16, IOS, Amsterdam, 2008.

\bibitem{WanNormTrace}
D. Wan, Norm-trace and Kloosterman sums in finite semisimple algebras, Frontiers in Combinatorics and Number Theory \textbf{1} (2026), 74--88.

\bibitem{ZelingherMatrix}
E. Zelingher, On matrix Kloosterman sums and Hall-Littlewood polynomials, Trans. Amer. Math. Soc. {\bf 378} (2025), no.~5, 3597--3623.
\end{thebibliography}
\end{document}